\documentclass{article}
\usepackage{amsmath,amsthm,amscd,amssymb,amsfonts}
\usepackage[mathscr]{eucal}
\usepackage{latexsym}

\usepackage{amsfonts}

\theoremstyle{plain}
\newtheorem{theorem}{Theorem}[section]
\newtheorem{lemma}{Lemma}[section]
\newtheorem{proposition}{Proposition}[section]
\newtheorem{corollary}{\bf {Corollary}}[section]
\theoremstyle{definition}

\newtheorem{remark}{{\bf Remark}}[section]
\theoremstyle{remark}

\newcommand{\vep}{\varepsilon}

\newcommand{\Sym}{\mathrm{Sym}}
\newcommand{\ord}{\mathrm{ord}}
\newcommand{\nd}{\mathrm{nd}}

\newcommand{\cala}{{\mathcal A}}
\newcommand{\calb}{{\mathcal B}}
\newcommand{\calc}{{\mathcal C}}

\newcommand{\calh}{{\mathcal H}}

\newcommand{\cals}{{\mathcal S}}

\newcommand{\CC}{{\mathbb C}}

\newcommand{\QQ}{{\mathbb Q}}
\newcommand{\RR}{{\mathbb R}}

\newcommand{\ZZ}{{\mathbb Z}}

\makeatletter 
\newcommand{\LEQQ}{\mathrel{\mathpalette\gl@align<}}
\newcommand{\GEQQ}{\mathrel{\mathpalette\gl@align>}}
\newcommand{\gl@align}[2]{\lower.6ex\vbox{\baselineskip\z@skip\lineskip\z@ 
\ialign{$\m@th#1\hfil##\hfil$\crcr#2\crcr=\crcr}}}
\makeatother

\def\mattwo(#1;#2;#3;#4){\left(\begin{matrix}
                               #1 & #2 \\
                               #3  & #4
                                      \end{matrix}\right)}
\title {A remark on $p$-adic  Siegel Eisenstein series}
\author{Hidenori Katsurada 
\thanks{Muroran Institute of Technology
27-1 Mizumoto, Muroran, Hokkaido  050-8585, Japan
\quad
and
Department of Mathematics, Hokkaido University,
Kita 10, Nishi 8, Kitaku, Sapporo, Hokkaido 060-0810, Japan,
hidenori@mmm.muroran-it.ac.jp,}
\and
 Shoyu Nagaoka \thanks{
Department of Mathematics, Yamato University, 
2-5-1 Katayama, Suita, Osaka 564-0082, Japan,
shoyu1122.sn@gmail.com }}

\date{}
\begin{document}
\maketitle

\begin{abstract}
A generalization of Serre's $p$-adic Eisenstein series in the case
of Siegel modular forms is studied
and a coincidence between a $p$-adic Siegel Eisenstein series and
a genus theta series associated with a quaternary quadratic form is proved.
\end{abstract}


\section{Introduction}

In \cite{Serre}, Serre defined the concept of a $p$-adic modular form as the $p$-adic limit of a
$q$-expansion of modular forms with rational Fourier coefficients. The $p$-adic Eisenstein series
was introduced as a typical example of $p$-adic modular forms, and its relation with the
modular forms on $\Gamma_0(p)$ was also studied. Later, Serre's $p$-adic Eisenstein series was 
extended to the case of Siegel modular groups, revealing various interesting properties. For example,
it was shown that some $p$-adic Siegel Eisenstein series becomes the usual Siegel modular form
of level $p$ (\cite{Nagaoka1}, \cite{Kat-Nag}).

Let $p$ be a prime number and $\{k_m\}$ the sequence defined by
$$
k_m:=2+(p-1)p^{m-1}.
$$
For sequence $\{k_m\}$, the $p$-adic Siegel Eisenstein series
$$
\widetilde{E}_2^{(n)}:=\lim_{m\to\infty}E_{k_m}^{(n)}
$$
is defined, where $E_k^{(n)}$ is the ordinary Siegel Eisenstein series of degree $n$ and  weight $k$.

Let $S^{(p)}$ be a positive definite, half-integral symmetric matrix of degree $4$ with
level $p$ and determinant $p^2/16$. We denote by genus\,$\Theta^{(n)}(S^{(p)})$ the genus theta series of
degree $n$ associated with $S^{(p)}$ (for the precise definition, see $\S$ \ref{genusT}). Genus theta
series genus\,$\Theta^{(n)}(S^{(p)})$ is a Siegel modular form of weight 2 on the level $p$ modular
group $\Gamma_0^{(n)}(p)$.

In \cite{KN}, Kikuta and the second author showed the coincidence between the two objects
$\widetilde{E}_2^{(2)}$ and genus\,$\Theta^{(2)}(S^{(p)})$. This result asserts that there is a correspondence
between some $p$-adic Siegel Eisenstein series and some genus theta series when the degree
is 2. The main purpose of this paper is to show that this coincidence still exists for any $n$.
Namely we prove the following theorem.
\vspace{3mm}
\\
\textbf{Theorem}\quad {\it We assume that $p$ is an odd prime number.
Then the degree $n$ $p$-adic Siegel Eisenstein series $\widetilde{E}_2^{(n)}$ coincides with
the degree $n$ genus theta series {\rm genus}\,$\Theta^{(n)}(S^{(p)})$:
$$
\widetilde{E}_2^{(n)}={\rm genus}\,\Theta^{(n)}(S^{(p)}).
$$
In particular, the $p$-adic Siegel Eisenstein series $\widetilde{E}_2^{(n)}$ is a Siegel
modular form of weight 2 on $\Gamma_0^{(n)}(p)$.
}
\\

The equality of this theorem is proved by showing that the Fourier coefficients on both
sides are equal. In particular, as can be seen from the discussion in the text, the computation
of local densities for $n=3$ or $4$ is essential
(cf. \S\; \ref{LC34}).

By considering the $p$-adic first and second approximation of $\widetilde{E}_2^{(n)}$,
we obtain the following results.
\vspace{2mm}
\\
\textbf{Corollary}\quad (1)\;\; {\it Assume that $p>n$. Then the modular form
$\widetilde{E}_2^{(n)}$ of level $p$ and weight 2 is congruent to the weight $p+1$
Siegel Eisenstein series $E_{p+1}^{(n)}$ mod $p$\;{\rm :}
$$
  \widetilde{E}_2^{(n)} \equiv E_{p+1}^{(n)} \pmod{p}.
$$
}
(2)\;\;{\it
Assume that $p\geq 3$. Then we have
$$
\varTheta (E_{p+1}^{(3)}) \equiv 0 \pmod{p}\quad and \quad \varTheta (E_{p^2-p+2}^{(4)}) \equiv 0 \pmod{p^2},
$$
where $\varTheta$ is the theta operator {\rm (cf. $\S$ \ref{ThetaOp})}.
}
\vspace{3mm}
\\
Statement (1) is motivated by Serre's 
result in the case of elliptic modular
forms: For any modular form $f$ of weight 2 on $\Gamma_0(p)$, there is a modular
form $g$ of level one and weight $p+1$ satisfying
$$
f \equiv g \pmod{p}.
$$
The result described in (2) is related to the theory of the mod $p$ kernels of theta operators.
The second congruence provides an example of a Siegel modular form contained in the mod $p^2$ kernel
of a theta operator.
\vspace{4mm}
\\
{\sc Notation.}  
Let $R$ be a commutative ring. We denote by $R^{\times}$ the unit group of $R$. 
We denote by $M_{mn}(R)$ the set of
$m \times n$-matrices with entries in $R.$ In particular put $M_n(R)=M_{nn}(R).$   Put $GL_m(R) = \{A \in M_m(R) \ | \ \det A \in R^\times \},$ where $\det
A$ denotes the determinant of a square matrix $A$. For an $m \times n$-matrix $X$ and an $m \times m$-matrix
$A$, we write $A[X] = {}^t X A X,$ where $^t X$ denotes the
transpose of $X$. Let $\text{Sym}_n(R)$ 
denote
the set of symmetric matrices of degree $n$ with entries in
$R.$ Furthermore, if $R$ is an integral domain of characteristic different from $2,$ let  $\calh_n(R)$ denote the set of half-integral matrices of degree $n$ over $R$, that is, $\calh_n(R)$ is the subset of symmetric
matrices of degree $n$ with entries in the field of fractions of $R$ whose $(i,j)$-component belongs to
$R$ or ${1 \over 2}R$ according as $i=j$ or not.  
 We say that an element $A$ of $M_n(R)$ is non-degenerate if $\det A \not=0$. For a subset $\cals$ of $M_n(R)$ we denote by $\cals^{\nd}$ the subset of $\cals$
consisting of non-degenerate matrices. If $\cals$ is a subset of $\Sym_n(\RR)$ with $\RR$ the field of real numbers, we denote by $\cals_{>0}$ (resp. $\cals_{\ge 0}$) the subset of $\cals$
consisting of positive definite (resp. semi-positive definite) matrices. 
We sometimes write $\Lambda_n$ (resp. $\Lambda_n^+$) instead of $\calh_n(\ZZ)$
(resp. $\calh_n(\ZZ)_{>0}$). 
The group  
$GL_n(R)$ acts on the set $\Sym_n(R)$ by 
$$
GL_n(R) \times \Sym_n(R) \ni (g,A) \longmapsto A[g] \in \Sym_n(R).
$$
\noindent 
Let $G$ be a subgroup of $GL_n(R).$ For a $G$-stable subset ${\mathcal B}$ of $\Sym_n(R)$ we denote by $\calb/G$ the set of equivalence classes of $\calb$ under the action of  $G.$ We sometimes use the same symbol $\calb/G$ to denote a complete set of representatives of $\calb/G.$ We abbreviate $\calb/GL_n(R)$ as $\calb/\!\!\sim$ if there is no fear of confusion. Let $G$ be a subgroup of $GL_n(R)$. Then two symmetric matrices $A$ and $A'$ with
entries in $R$ are said to be $G$-equivalent  with each
other and write $A \sim_{G} A'$ if there is
an element $X$ of $G$ such that $A'=A[X].$ We also write $A \sim A'$ if there is no fear of confusion. 
For square matrices $X$ and $Y$ we write $X \bot Y = \begin{pmatrix} X &O \\ O & Y \end{pmatrix}$.

We put ${\bf e}(x)=\exp(2 \pi \sqrt{-1} x)$ for $x \in \CC,$ and for a prime number $q$ we denote by ${\bf e}_q(*)$ the continuous additive character of $\QQ_q$ such that ${\bf e}_q(x)= {\bf e}(x)$ for $x \in \ZZ[q^{-1}].$
 For a prime number $q$ we denote by $\ord_q(*)$ the additive valuation of $\QQ_q$ normalized so that $\ord_q(q)=1$. 
Moreover for any element $a, b \in \ZZ_q$
we write $b \equiv a \pmod {q}$ if $\ord_q(a-b) >0$.
\section{Siegel Eisenstein series and genus theta series}
\subsection{Siegel modular forms}
Let $\mathbb{H}_n$ be the Siegel upper-half space of degree $n$; then the Siegel
modular group $\Gamma^{(n)}:=Sp_n(\mathbb{R})\cap M_{2n}(\mathbb{Z})$ acts
discontinuously on $\mathbb{H}_n$. For a congruence subgroup $\Gamma'$ of
$\Gamma^{(n)}$, we denote by $M_k(\Gamma')$ the corresponding space of Siegel
modular forms of weight $k$ for $\Gamma'$. Later we mainly deal with the case $\Gamma'=\Gamma^{(n)}$
or $\Gamma_0^{(n)}(N)$ where
$$
\Gamma_0^{(n)}(N)=\left\{ \binom{A\;B}{C\;D}\in\Gamma^{(n)}\;\big{|}\; C \equiv O_n \pmod{N}\;\right\}.
$$
In both cases, $F\in M_k(\Gamma')$ has a Fourier expansion of the form
$$
F(Z)=\sum_{0\leq T\in\Lambda_n}a(F,T)\,{\bf e}(\text{tr}(TZ)).
$$

Taking $q_{ij}:={\bf e}(z_{ij})$
with $Z=(z_{ij})\in\mathbb{H}_n$, we write
$$
q^T:={\bf e}(\text{tr}(TZ))
=\prod_{1\leq i<j\leq n}q_{ij}^{2t_{ij}}\prod_{i=1}^nq_i^{t_i},
$$
where $T=(t_{ij})$ and $q_i=q_{ii}$,\,$t_i=t_{ii}$\,$(i=1,\cdots,n)$. Using this notation, we obtain
\begin{align*}
F=\sum_{0\leq T\in\Lambda_n}a(F,T)\,q^T  & =\sum_{t_i}\left(\sum_{t_{ij}}a(F,T)\prod_{i<j}q_{ij}^{2t_{ij}}\right)
                                                                       \prod_{i=1}^nq_i^{t_{ij}} \\
                                                                 & \in\mathbb{C}[q_{ij}^{-1},q_{ij}][\![q_1,\ldots,q_n]\!].      
\end{align*}

For a subring $R\subset\mathbb{C}$, we denote by $M_k(\Gamma')_R$ the set consisting of
modular forms $F$ all of whose Fourier coefficients $a(F,T)$ lie in $R$. Therefore, an element
$F\in M_k(\Gamma')_R$ may be regarded as an element of
$$
R[q_{ij}^{-1},q_{ij}][\![q_1,\ldots,q_n]\!]. 
$$
\subsection{Siegel Eisenstein series}
Define
$$
\Gamma_\infty^{(n)}:=\left\{\;\binom{A\;B}{C\;D}\in\Gamma^{(n)}\;\big{|}\; C=O_n\; \right\}.
$$
For an even integer $k>n+1$, define a series by
$$
E_k^{(n)}(Z)=\sum_{\binom{*\;*}{C\;D}\in\Gamma_\infty^{(n)}\backslash \Gamma^{(n)}}
                                                 \text{det}(CZ+D)^{-k},\quad Z\in\mathbb{H}_n.
$$
This series is an element of $M_k(\Gamma^{(n)})_{\mathbb{Q}}$  called the {\it Siegel Eisenstein series} of
weight $k$ for $\Gamma^{(n)}$.
\subsection{Genus theta series}
\label{genusT}
Fix $S\in\Lambda_{m}^{+}$ and define
$$
\theta^{(n)}(S;Z):=\sum_{X\in M_{mn}(\mathbb{Z})}\,{\bf e}(\text{tr}(S[X]Z)),\quad Z\in\mathbb{H}_n.
$$

Let $\{S_1,\ldots, S_d\}$ be a set of representatives of $GL_m(\mathbb{Z})$-equivalence classes
in $\text{genus}\,(S)$. The {\it genus theta series} associated with $S$ is defined by
$$
\text{genus}\,\Theta^{(n)}(S)(Z):=
\left(\sum_{i=1}^d\frac{\theta^{(n)}(S_i;Z)}{a(S_i,S_i)}\right)\Big{/}
\left( \sum_{i=1}^d\frac{1}{a(S_i,S_i)}\right), \quad Z\in\mathbb{H}_n           
$$
where
$$
a(S_i,S_i):=\sharp\{\;X\in M_m(\mathbb{Z})\;|\;S_i[X]=S_i\;\}\quad (\text{cf. \S\;\ref{LC34})}.
$$
\subsection{$p$-adic Siegel Eisenstein series}
Let $\{k_m\}_{m=1}^{\infty}$ be an increasing sequence of even positive integers which
is $p$-adically convergent.

If the corresponding sequence of Siegel Eisenstein series
$$
\{E_{k_m}^{(n)}\}\subset \mathbb{Q}[q_{ij}^{-1},q_{ij}][\![q_1,\ldots,q_n]\!]
$$
converges $p$-adically to an element of $\mathbb{Q}_p[q_{ij}^{-1},q_{ij}][\![q_1,\ldots,q_n]\!]$,
then we call the limit
$$
\lim_{m\to\infty}E_{k_m}^{(n)}\in \mathbb{Q}_p[q_{ij}^{-1},q_{ij}][\![q_1,\ldots,q_n]\!]
$$
a {\it $p$-adic Siegel Eisenstein series}.

\section{Main result}
\subsection{Statement of the main theorem}
As stated in the Introduction, the main result of this paper asserts a coincidence between
some $p$-adic
Siegel Eisenstein series and some genus theta series.

First we consider an element $S^{(p)}$ of $\Lambda_4^+$ with level $p$ and
determinant $p^2/16$.
(The existence of such a matrix will be proved  in Lemma \ref{lem.existence-of-S}.)

Next we consider a special $p$-adic Siegel Eisenstein series.

Let $\{k_m\}_{m=1}^\infty\subset \mathbb{Z}_{>0}$ be a sequence defined by
$$
k_m=k_m(p):=2+(p-1)p^{m-1}\qquad (p:\;\text{prime}).
$$
This sequence converges $p$-adically to 2.
We associate this sequence with
the sequence of Siegel Eisenstein series
$$
\{E_{k_m}^{(n)}\}_{m=1}^\infty\subset \mathbb{Q}[q_{ij}^{-1},q_{ij}][\![q_1,\ldots,q_n]\!].
$$
As we prove in the following, this sequence defines a $p$-adic Siegel Eisenstein series.
We set
$$
\widetilde{E}_2^{(n)}:=\lim_{m\to\infty}E_{k_m}^{(n)}.
$$
Our main theorem can be stated as follows.
\begin{theorem}
\label{statementmain}
Let $p$ be an odd prime number and $S^{(p)}$ be as above.
Then the following identity holds:
$$
\widetilde{E}_2^{(n)}={\rm genus}\,\Theta(S^{(p)})
$$
\end{theorem}

If $n=2$, the above identity has already been proved in \cite{KN},
and the proof for a general $n$ is essentially the
case $n=3$ or $4$, which will be presented in the next section.

\subsection{Case $\boldsymbol{n=3}$ or $\boldsymbol{4}$}
\label{n=3-4}
We prove our main result in the case that $n$ is $3$ or $4$.
\begin{theorem} \label{th.main-result}
\begin{itemize}
\item[{\rm (1)}] For any $T \in \Lambda_3^+$, we have 
 \begin{align*}
a(\widetilde E_2^{(3)},T)=a(\mathrm{genus} \ \Theta^{(3)}(S^{(p)}),T).
\end{align*}
\item[{\rm (2)}] For any $T \in \Lambda_4^+$, we have 
 \begin{align*}
a(\widetilde E_2^{(4)},T)=a(\mathrm{genus} \ \Theta^{(4)} (S^{(p)}),T).
\end{align*}
\end{itemize}
\end{theorem}
By the above theorem and \cite{KN}. we have the following result.
\begin{corollary}
\label{cor.main-result2}
Let $n=3$ or $4$. Then
\begin{align*} 
\widetilde E_2^{(n)}=\mathrm{genus} \ \Theta^{(n)}(S^{(p)}).
\end{align*}
\end{corollary}

To prove Theorem \ref{th.main-result}, we compute the both-hand 
sides of the equality in the theorem  explicitly. 
We also use the notation in \cite{IK22}. 
Let $q$ be a prime number. 

Let $\langle  \ {} \ , \ {} \ \rangle=\langle \ {} \ , \ {} \ \rangle_Q$ be the Hilbert symbol on  $\QQ_q$. 
Let $T$ be  a non-degenerate symmetric matrix with entries in $\QQ_q$ of degree $n$. Then $T$ is $GL_n(\QQ_q)$-equivalent to $b_1 \bot \cdots \bot b_n$ with $b_1,\ldots,b_n \in \QQ_q^{\times}$. 
Then we define the Hasse invariant  $h_q(T)$ as  
\[h_q(T)=\prod_{1 \le i \le j \le n} \langle b_i,b_j \rangle.\]
We also define $\varepsilon_q(T)$ as 
\[\varepsilon_q(T)=\prod_{1 \le i < j \le n} \langle b_i,b_j \rangle.\]
These do not depend on the choice of $b_1,\ldots,b_n$. We also denote by $\eta_q(T)$ the Clifford invariant of $T$
(cf. \cite{IK22}). 
Then we have 
\[\eta_q(T)=
\begin{cases}
\langle -1,-1 \rangle^{m(m+1)/2} \langle (-1)^m,\det (T) \rangle \varepsilon_q(T) & \ \text{if $n=2m+1$,}\\
\langle -1,-1 \rangle^{m(m-1)/2} \langle (-1)^{m+1},\det (T) \rangle \varepsilon_q(T) & \ \text{if $n=2m$,}
\end{cases}\]
and hence, 
\[\eta_q(T)=
\begin{cases}
\langle -1,-1 \rangle^{m(m+1)/2} \langle (-1)^m \det (T),\det (T) \rangle \varepsilon_q(T) & \ \text{if $n=2m+1$,}\\
\langle -1,-1 \rangle^{m(m-1)/2} \langle (-1)^{m+1} \det (T),\det (T) \rangle \varepsilon_q(T) & \ \text{if $n=2m$.}
\end{cases}\]
Let $T \in \Sym_n(\QQ_q)$. Then, by the product formula for the Hilbert symbol, we have
$\prod_q h_q(T)=1$ and 
\[\prod_q \eta_q(T)=\begin{cases} (-1)^{(n^2-1)/8} & \text{ if } n \text{ is odd} ,\\
(-1)^{n(n-2)/8} & \text{ if } n \text{ is even}.
\end{cases}\]
For $a \in \ZZ_q, a \not=0$, define $\chi_q(a)$ as
\[
\chi_q(a)=
\begin{cases} 
1 & \text{ if } \QQ_q(\sqrt{a})=\QQ_q,\\
-1 & \text{ if } \QQ_q(\sqrt{a})/\QQ_q \text { is unramified quadratic}, \\
0 & \text{ if } \QQ_q(\sqrt{a})/\QQ_q \text{  is  ramified quadratic}.
\end{cases}
\]

We now prove the exisitence of the matrix $S^{(p)}$ that appears in the main theorem.
\vspace{2mm}
\\
Let $H_k = \overbrace{H \bot ...\bot H}^k$ with 
$H =\left(\begin{matrix}
             0        & 1/2  \\
             1/2        & 0 
                                   \end{matrix}\right)$.
We take an element $\epsilon \in \ZZ_p^\times$ such that $\chi_p(-\epsilon)=-1$, 
and put $U_0=1 \,\bot\, \epsilon$.
\begin{lemma}\label{lem.existence-of-S}
For each prime number  $q$, let $S_q$ be an element of $\calh_4(\ZZ_q)$ 
such that
\begin{align*} 
S_q =\begin{cases} U_0 \bot pU_0 & \text{ if }  q=p, \\
H_2 & \text{ otherwise}.
\end{cases}
\end{align*}
Then there exists an element $S$ of $\Lambda_4^+$ with level $p$ and determinant $16^{-1}p^2$ such that
\begin{align*}
S \sim_{GL_4(\ZZ_q)} S_q \text{ for any } q. \tag{E}
\end{align*} 
Conversely, if $S$ is an element of $\Lambda_4^+$ with level $p$ and determinant $16^{-1}p^2$ , then $S$ satisfies  condition {\rm (E)}.
\end{lemma}
\begin{proof}
Let $S_q$ be as above.  By assumption, we have
\[ (16^{-1}p^2)^{-1}\det (S_q) \in (\ZZ_q^\times)^2 \text{ for any } q,\]
and
\begin{align*} h_q(S_q)=\begin{cases} -1 & \text{ if } q=p \text{ or } 2, \\
1 & \text{ otherwise}.
\end{cases} 
\end{align*}
 Hence, by \cite[Theorem 4.1.2]{Ki2}, there exists an element $S$ of $\text{Sym}_4(\QQ)_{>0}$ satisfying condition (E). By construction, we easily see that $S$ belongs to $\calh_4(\ZZ)_{>0}$ with level $p$ and determinant  $16^{-1}p^2$.

Conversely, let  $S$ be  an element of $\Lambda_4^+$ with level $p$ and determinant $16^{-1}p^2$. 
Then, we have
\[\det (2S) \in  (\ZZ_q^\times)^2 \text{ for any } q\not=p.\]
Hence, we have $S \sim_{GL_4(\ZZ_q)} H_2$,
so we have  $h_q(S)=1$ if $q \not=p,\,2$ and $h_2(S)=-1$.
Thus, we have $h_p(S)=-1$. By definition, we have 
$p^{-2}\det (S) \in (\ZZ_p^\times)^2$ and  $pS^{-1} \in \calh_4(\ZZ_p)$. Then, we easily see that 
$S \sim_{GL_4(\ZZ_p)} U_0 \bot pU_0$.
This proves the assertion.
\end{proof}

\subsubsection{Computation of $\boldsymbol{a(\widetilde{E}_2^{(n)},T)}$ for $\boldsymbol{n=3}$ or $\boldsymbol{4}$}

In the case $n$ is even, for $T \in \calh_n(\ZZ_q)^{\rm nd}$ we put
$\xi_q(T)=\chi_q((-1)^{n/2} \det (2T))$.
For $T \in \calh_n(\ZZ_q)^{\rm nd}$, let $b_q(T,s)$ be the Siegel series of $T$. 
Then $b_q(T,s)$ is a polynomial in $q^{-s}$. More precisely, we define a polynomial $\gamma_q(T,X)$ in $X$ by 
$$\gamma_q(T,X)=
\begin{cases}
(1-X)\prod_{i=1}^{n/2}(1-q^{2i}X^2)(1-q^{n/2}\xi_q(T) X)^{-1} & \text{ if $n$ is  even,} 
\vspace{2mm}
\\ 
(1-X)\prod_{i=1}^{(n-1)/2}(1-q^{2i}X^2) & \text{ if $n$ is  odd.} \end{cases}$$ 
Then  there exists a polynomial $F_q(B,X)$ in $X$ with coefficients in $\ZZ$ such that 
$$F_q(T,q^{-s})={b_q(T,s) \over \gamma_q(T,q^{-s})}.$$ 
The properties of the polynomial $F_q(B,X)$ is studied in detail by the first author
in \cite{K99}. (In particular,  see \cite[Theorem 3.2]{K99} for the properties used below.)

 We  put
\[\widetilde b_q(T,X)=\gamma_q(T,X)F_q(T,X).\]
\begin{proposition}\label{prop.FC-Siegel}
For any $T \in \Lambda_n^+$, we have
\begin{align*}
a(E_k^{(n)},T)&= (-1)^{[(n+1)/2]}2^{n-[n/2]} \frac{k}{B_k} \prod_{i=1}^{[n/2]}\frac{2k-2i}{B_{2k-2i}} 
                    \prod_q F_q(T,q^{k-n-1})\\
&\times \begin{cases} \frac{B_{k-n/2,\chi_T}}{k-n/2} & \text{ if } n \text{ is even,} \\
1 & \text{ if } n \text{ is odd,}
\end{cases}
\end{align*}
where $\chi_T$ denotes the primitive Dirichlet character corresponding to the extension
$\mathbb{Q}(\sqrt{(-1)^{n/2}{\rm det}(2T)})/\mathbb{Q}$.
\end{proposition}

\begin{remark}\label{rem.comariosn-KN}
Let $F_q^{(n)}(T,X)$ be the polynomial in \cite[Theorem 2.2]{Kat-Nag}.
Then it coincides with $F_q(T,X)$ if $n$ is even. But
it is $\eta_q(T)F_q(T,X)$ if $n$ is odd.
\end{remark}

 \begin{proposition} \label{prop.limit-of-Kitaoka-polynomial}
Let $T \in \calh_n(\ZZ_q)^{\rm nd}$.
\begin{itemize}
\item[{\rm (1)}] Let $q \not=p$. Then
we have
\begin{align*}
&\lim_{m\to\infty} F_q(T,q^{k_m-n-1})=F_q(T, q^{2-n-1}).\\
\end{align*}
\item[{\rm (2)}]  We have 
\begin{align*}
&\lim_{m \to\infty} F_p(T,p^{k_m-n-1})=1
\end{align*}
\end{itemize}
\end{proposition} 
\begin{proof}
By construction, $F_q(T,X)$ belongs to $\ZZ[X]$ and its first coefficient is $1$. 
Thus the assertion holds.
\end{proof} 
\begin{lemma} \label{lem.vanishing-of-F}
Let $T \in \calh_3(\ZZ_q)^{\rm nd}$.
Suppose that $\eta_q(T)=-1$. Then
\[F_q(T,q^{-2})=0.\]
\end{lemma}
\begin{proof} By the functional equation of $F_q(T,X)$ (cf. \cite[Theorem 3.2]{K99}), we have
\begin{align*}
F_q(T,q^{-4}X^{-1}) & =\eta_q(T)(q^2X)^{\mathrm{ord}_q(2^{2}\det T)}F_q(T,X)\\
& =-(q^2X)^{\mathrm{ord}_q(2^{2}\det T)}F_q(T,X).
\end{align*}
Hence, we have 
\begin{align*}
F_q(T,q^{-2})=-F_q(T,q^{-2}).
\end{align*}
Therefore, the assertion holds.
\end{proof}
Here we prepare some result on Bernoulli numbers.
\begin{lemma}\;{\rm (Carlitz \cite[{\sc Theorem 3}]{Carlitz1})}
\label{lem.limit-of-Bernoulli-number}
Assume that $p$ is an odd prime number.  
For $t=rp^k(p-1)$\,{\rm (}$r\in\mathbb{Z}_{>0},\,k\in\mathbb{Z}_{\geq 0}${\rm )}, 
the numerator of
$\displaystyle B_t+\frac{1}{\,p\,}-1$ is divisible by $p^k$. In particular
$$
\lim_{k\to\infty}B_{r(p-1)p^k}=1-\frac{1}{\,p\,}.
$$
for any $r\in\mathbb{Z}_{>0}$.
\end{lemma}

Now we obtain the following theorem. 
\begin{theorem} \label{th.limit-of-FC-of-Eisenstein}
\begin{itemize}
\item[{\rm (1)}] Let $T \in \Lambda_3^+$.
Then, we have
\begin{align*}
a(\widetilde E_2^{(3)},T)=\frac{576}{(1-p)^2} \prod_{q \not=p} F_q(T,q^{-2}),
\end{align*}
and in particular if  $\eta_p(T)=1$, we have
\begin{align*}
a(\widetilde E_2^{(3)},T)=0.
\end{align*}
\item[{\rm (2)}] Let $T \in \Lambda_4^+$.
\begin{itemize}
\item[{\rm (2.1)}] Suppose that $\det (2T)$ is not square. Then
\[a(\widetilde E_2^{(4)},T)=0.\]
\item[{\rm (2.2)}] Suppose that $\det (2T)$ is  square. Then
\begin{align*}
a(\widetilde E_2^{(4)},T)=\frac{1152}{(1-p)^2} \prod_{q \not=p} F_q(T,q^{-3}).
\end{align*}
In particular if $\eta_p(T)=1$ or  $\mathrm{ord}_p(\det(2T)) =0$, then
\[a(\widetilde E_2^{(4)},T)=0.\]
\end{itemize}
\end{itemize}
\end{theorem}
\begin{proof}
(1) We have
\begin{align*} (\star)\qquad\qquad
\lim_{m\to\infty} \frac{B_{k_m}}{k_m}=
\lim_{m\to\infty} \frac{B_{2k_m-2}}{2k_m-2}=
\frac{(1-p)B_2}{2}=\frac{1-p}{12}.
\end{align*}
These identities can be proved by Kummer's congruence when $p>3$.
For $p=3$, these identities are shown using the extended Kummer's congruence
(cf. \cite[COROLLAIRE 3]{Fr}). This proves the first part of the assertion.
Suppose that $\eta_p(T)=1$.
We note that 
\[\prod_q \eta_q(T)=-1.\]
Therefore, if $\eta_p(T)=1$, there is a prime number $q \not=p$ such that $\eta_q(T)=-1$. Therefore, by Lemma \ref{lem.vanishing-of-F}, we have $F_q(T,q^{-2})=0$. This proves the second part of the assertion.

(2)  Put
\[\calb_{T,m}=\Big(\frac{2k_m-4}{B_{2k_m-4}} \Big)\Big(\frac{B_{k_m-2,\chi_T}}{k_m-2}\Bigr).\]
Then 
\[\calb_{T,m} =\Big(\frac{2(p-1)p^{m-1}}{B_{2(p-1)p^{m-1}}} \Big)
                                           \Big(\frac{B_{(p-1)p^{m-1},\chi_T}}{(p-1)p^{m-1}}\Bigr).\]
First suppose that $\det (2T)$ is not square. 
Then, $\chi_T$ is non-trivial, and $\frac{B_{(p-1)p^{m-1},\chi_T}}{(p-1)p^{m-1}}$ is  $p$-integral 
for any $m \ge 1$ (cf. \cite{Carlitz2}). Hence, by the theorem of von Staudt-Clausen, 
\[\lim_{m\to\infty} \calb_{T,m}=0.\]
This proves the assertion (2.1).
Next suppose that $\det (2T)$ is square. Then, $B_{(p-1)p^{m-1},\chi_T}=B_{(p-1)p^{m-1}}$, and
\[\calb_{T,m} =2\frac{B_{(p-1)p^{m-1}}}{B_{2(p-1)p^{m-1}}}.\]
By Lemma \ref{lem.limit-of-Bernoulli-number}, 
\[\lim_{m\to\infty} \calb_{T,m}=2\cdot\lim_{m\to\infty}\frac{B_{(p-1)p^{m-1}}}{B_{2(p-1)p^{m-1}}}=2.\]
This fact, together with $(\star)$, proves the first part of the assertion (2.2).
To prove the remaining part, first suppose that $\eta_p(T)=1$.
We note that 
\[\prod_q \eta_q(T)=-1.\]
Therefore, there is a prime number $q \not=p$ such that $\eta_q(T)=-1$. 
We have $F_q(T,q^{-3})=0$ as will be proved in Corollary \ref{add-cor2-2}.
Next suppose that $\mathrm{ord}_p(\det (2T)) =0$. Then, $\eta_p(T)=1$, and the assertion follows from the above. 
\end{proof}
\subsubsection{Computation of
 $\boldsymbol{a(\text{genus}\,\Theta^{(n)}(S^{(p)}),T)}$ for $\boldsymbol{n=3}$ or $\boldsymbol{4}$}
 \label{LC34}
For $S \in \calh_m(\ZZ_q)^{\nd}$ and $T \in \calh_n(\ZZ_q)^{\nd}$ with $m \ge n$, we define the local density 
$\alpha_q(S,T)$ as
\[\alpha_q(S,T)=2^{-\delta_{m,n}}\lim_{e\to\infty} q^{e(-mn+n(n+1)/2)}\cala_e(S,T),
\]
where 
\begin{align*}
\cala_e(S,T)=\{X \in M_{mn}(\ZZ_q)/q^eM_{mn}(\ZZ_q) \ | \ S[X] \equiv T\!\pmod{ q^e \calh_n(\ZZ_q)}\,\},
\end{align*}
and $\delta_{m,n}$ is the Kronecker delta.

We say that an element $X$ of $M_{mn}(\ZZ_q)$ with $m \ge n$ is
{\it primitive} if $\mathrm{rank}_{\ZZ_q/q\ZZ_q} X=n$.
We also say that an element $X \mod q^e$ of $M_{mn}(\ZZ_q)/q^eM_{mn}(\ZZ_q)$ is {\it primitive} if $X$ is primitive. 
This definition does not depend on the choice of $X$. 
From now on, for $X \in M_{mn}(\ZZ_q)$, we often use the same symbol $X$ to denote the class of $X$ mod $q^e$. 
For  $S \in \calh_m(\ZZ_q)^{\nd}$ and $T \in \calh_n(\ZZ_q)$ with $m \ge n$, we define the primitive local density $\beta_q(S,T)$ as 
\[\beta_q(S,T)=2^{-\delta_{m,n}}\lim_{e\to\infty} q^{e(-mn+n(n+1)/2)}\calb_e(S,T),\]
where
\[\calb_e(S,T)=\{X \in \cala_e(S,T) \ | \ X \text{ is primitive}\}.
\]

The following is due to \cite{Ki1}.
\begin{lemma} \label{lem.local-density-via-primitive-local-density}
Let $S \in \calh_m(\ZZ_q)^{\rm nd}$ and $T \in \calh_n(\ZZ_q)^{\rm nd}$ with $m \ge n$. Then
\[\alpha_q(S,T)=\sum_{g \in GL_n(\ZZ_q) \backslash M_n(\ZZ_q)^{\rm nd}} q^{(-m+n+1) \det g}\beta_q(S,T[g^{-1}]).\]
\end{lemma}
For $S \in \Sym_m(\ZZ_q)^{\nd}$ and $T \in \Sym_n(\ZZ_q)^{\nd}$  we also define another local density $\widetilde \alpha_q(T,S)$ as 
\begin{align*}
\widetilde \alpha_q(T,S)=2^{-\delta_{m,n}}\lim_{e\to\infty} q^{e(-mn+n(n+1)/2)}\widetilde \cala_e(T,S),
\end{align*}
where 
\begin{align*}
\widetilde \cala_e(T,S)=\{X \in M_{mn}(\ZZ_q)/q^eM_{mn}(\ZZ_q) \ | 
                                          \ S[X] \equiv T \!\pmod{q^e \mathrm{Sym}_n(\ZZ_q)}\,\}.
\end{align*}
\begin{remark} \label{rem.comparison-densities}
We note that
\begin{align*}
\widetilde \alpha_q(2T,2S)=2^{m\delta_{2,p}}\alpha_q(S,T)
\end{align*}
for $S \in \calh_m(\ZZ_q)$ and $T \in \calh_n(\ZZ_q)$.
\end{remark}
\vspace{3mm}
For $S \in \Lambda_m^+$ and $T \in \Lambda_n^+$, put
\[a(S,T)=\#\{X \in M_{mn}(\ZZ) \ | \ S[X]=T\}.\]
Moreover put 
\[M(S)=\sum_{i=1}^{d} {1 \over a(S_i,S_i)} \]
and 
\[\widetilde a(S,T)=M(S)^{-1}\sum_{i=1}^d {a(S_i,T) \over a(S_i,S_i)},\]
where $S_i$ runs over a complete set of $GL_m(\ZZ)$-equivalence classes in $\text{genus}(S)$.
Then we have the following formula.
\begin{proposition} \label{prop.Siegel-main-theorem}
Under the above notation, we have
\begin{align*}
&\widetilde a(S,T)=\\
&2^n \vep_{n,m}\pi^{n(2m-n+1)/4}\prod_{i=1}^{n-1} \Gamma((m-i)/2)^{-1}(\det (2S))^{-n/2}(\det (2T))^{(m-n-1)/2}\\
&\times \prod_q \alpha_q(S,T),
\end{align*}
where
\[\vep_{n,m}=\begin{cases} 1/2  & \text{ if  either } m=n+1 \text{ or } m=n>1, \\
1 & \text{otherwise}. \end{cases}\]
\end{proposition}
\begin{proof}
We note that we have $\widetilde a(S,T)=\widetilde a(2S,2T)$. By \cite[Theorem 6.8.1]{Ki2} we have 
\begin{align*}
&\widetilde a(2S,2T)\\
&=\vep_{n,m}\pi^{n(2m-n+1)/4}\prod_{i=1}^{n-1} \Gamma((m-i)/2)^{-1}(\det (2S))^{-n/2}(\det (2T))^{(m-n-1)/2}\\
&\times \prod_q \widetilde \alpha_q(2T,2S).
\end{align*}
Then, the assertion follows from Remark \ref{rem.comparison-densities}.

\end{proof}


\begin{proposition} \label{prop.mass-formula}
\begin{itemize}
\item[{\rm (1)}] Let $T \in \Lambda_3^+$. Then we have
\[a(\mathrm{genus} \ \Theta^{(3)}(S^{(p)}),T)=8p^{-3}\pi^{4} \alpha_p(U_0 \bot pU_0,T)\prod_{q \not=p} \alpha_q(H_2,T).\]
\item[{\rm (2)}] Let $T \in \Lambda_4^+$. Then we have
\[a(\mathrm{genus}\  \Theta^{(4)}(S^{(p)}),T)=16p^{-4}\pi^{4} \det(2T)^{-1/2}\alpha_p(U_0 \bot pU_0,T)\prod_{q \not=p} \alpha_q(H_2,T).\]
\end{itemize}
\end{proposition}
\begin{proof}
(1) By Proposition \ref{prop.Siegel-main-theorem}, we have 
\[a(\mathrm{genus} \ \Theta^{(3)}(S^{(p)}),T)=8p^{-3}\pi^{4} \prod_{q } \alpha_q(S^{(p)},T).\]
By Lemma \ref{lem.existence-of-S}, we prove the assertion.

(2) Again by Proposition \ref{prop.Siegel-main-theorem}, we have 
\[a(\mathrm{genus} \ \Theta^{(4)}(S^{(p)}),T)=16p^{-4}\pi^{4} \det(2T)^{-1/2}\prod_{q } \alpha_q(S^{(p)},T).\]
Again by Lemma \ref{lem.existence-of-S}, we prove the assertion.
\end{proof}

The following lemma has been essentially proved in \cite[Lemma 14.8]{Sh1}.

\begin{lemma} \label{lem.local-density-and-Siegel-series}
Let $T \in \calh_n(\ZZ_q)^{\rm nd}$ with $k \ge n/2$. Then, we have
\[\alpha_q(H_k,T)=2^{-\delta_{2k,n}}\widetilde b_q(T,q^{-k}).\]
\end{lemma}

\begin{proposition}\label{prop.local-density-at-q}
\label{prop2-5}
Let $q \not=p$. 
\begin{itemize}
\item[{\rm (1)}] Let $T \in \calh_3(\ZZ_q)^{\rm nd}$. Then,
\[\alpha_q(H_2,T)=(1-q^{-2})^2 F_q(T,q^{-2}).\]
\item[{\rm (2)}] 
Let $T \in \calh_4(\ZZ_q)^{\rm nd}$ and suppose that $\det(2T)$ is square in $\ZZ_q$. Then,
\[\alpha_q(H_2,T)=(1-q^{-2})^2 q^{\mathrm{ord}_q(\det (2T))/2}F_q(T,q^{-3}).\]
\end{itemize}
\end{proposition}
\begin{proof}
(1) By definition, $\gamma_q(T,X)=(1-X)(1-q^{2}X^2)$. Thus the assertion is easy to prove using Lemma \ref{lem.local-density-and-Siegel-series}.

(2) Again by definition, $\gamma_q(T,X)=(1-X)(1-q^{2}X^2)(1+q^2X)$. Hence, by Lemma \ref{lem.local-density-and-Siegel-series}, we have
\[\alpha_q(H_2,T)=(1-q^{-2})^2 F_q(T,q^{-2}).\]
Thus the assertion follows from the functional equation of $F_q(T,X)$.
\end{proof}

\begin{corollary}
\label{add-cor2-2} Assume that $\det (2T)$ is square in $\mathbb{Z}_q$.
Let $T \in \mathcal{H}_4(\mathbb{Z}_q)^{\mathrm{nd}}$ and suppose that $\eta_q(T)=-1$.
Then $F_q(T,q^{-3})=0$.
\end{corollary}
\begin{proof}
Since $\eta_q(T)=1$, we have $\alpha_q(H_2,T)=0$. Thus the assertion follows from
Proposition \ref{prop2-5} (2).
\end{proof}

The following proposition is one of key ingredients in the proof of our main result.

\begin{proposition}\label{prop.local-density-at-p}
\begin{itemize}
\item[{\rm (1)}] Let $T \in \calh_3(\ZZ_p)^{\nd}$. Then
\[\alpha_p(U_0 \bot pU_0,T)=(1+p)(1+p^{-1})(1-\eta_p(T)).\]
\item[{\rm (2)}] Let $T \in \calh_4(\ZZ_p)^{\nd}$. 
If $\alpha_p(U_0 \bot pU_0,T) \not=0$, then $\det (2T) =p^{2m}\xi^2$ with $m \in \ZZ_{>0}, \xi \in \ZZ_p^\times$ and $\eta_p(T)=-1$. 
Conversely, if $T$ satisfies this condition, then 
\begin{align*}
\alpha_p(U_0 \bot pU_0, T)=2(1+p^{-1})^2 p^m.
 \end{align*}
\end{itemize}
\end{proposition}
The above proposition may be proved by the explicit formula in \cite[THEOREM]{SH}. But to apply the formula, we have to compute  too many quantities associated with $U_0 \bot pU_0$ and $T$. 
Therefore, we here use another method. 

We say that $S \in \calh_n(\ZZ_q)^{\rm nd}$ is {\it maximal} if
there is no element 
$g \in M_n(\ZZ_q)^{\mathrm{nd}}$ such that $\det g \in q\ZZ_q$ and  $S[g^{-1}] \in \calh_n(\ZZ_q)$. 

First we prove Proposition \ref{prop.local-density-at-p} (1).
The following lemma is easy to prove.

\begin{lemma}
\label{lem.maximal-element3}
Let $q$ be an odd prime number.
Suppose that $T \in \calh_3(\ZZ_q)^{\rm nd}$ is  maximal. Then
$T$ is one of the following matrices.
\begin{itemize}
\item[{\rm (1)}] $T \sim_{GL_3(\ZZ_q)} \epsilon_1 \bot \epsilon_2 \bot \epsilon_3$ with $\epsilon_i \in \ZZ_q^\times \ (i=1,2,3)$.
\item[{\rm (2)}] $T \sim_{GL_3(\ZZ_q)} \epsilon_1 \bot \epsilon_2 \bot q\epsilon_3$ with $\epsilon_i \in \ZZ_q^\times \ (i=1,2,3)$.
\item[{\rm (3)}] $T \sim_{GL_3(\ZZ_q)} \epsilon_1 \bot q\epsilon_2 \bot q\epsilon_3$ with $\epsilon_i \in \ZZ_q^\times \ (i=1,2,3)$ such that $\chi_q(-\epsilon_2\epsilon_3)=-1$.
\end{itemize}
Moreover, we have
\[ \eta_q(T)=\begin{cases}1 & \text{ in the case {\rm (1)},} \\ 
\chi_q(-\epsilon_1\epsilon_2) &  \text{ in the case {\rm (2)},} \\
-1 & \text{ in the case {\rm (3)}}.
\end{cases}\]
\end{lemma}
\begin{lemma}
\label{lem.explicit-primitive-locla-density3}
Let $T \in \calh_3(\ZZ_p)^{\rm nd}$.
\begin{itemize}
\item[{\rm (1)}] Suppose that $T$ is non-maximal. Then $\beta_p(U_0 \bot pU_0,T)=0$.
\item[{\rm (2)}] Suppose that $T$ is maximal and $\eta_p(T)=1$. Then 
$\beta_p(U_0 \bot pU_0,T)=0$.
\item[{\rm (3)}] Suppose that $T$ is maximal and $\eta_p(T)=-1$. Then \[\beta_p(U_0 \bot pU_0,T)=2(1+p)(1+p^{-1}).\]
\end{itemize}
\end{lemma}
\begin{proof}
We may suppose that $T=\epsilon_1 p^{a_1} \bot \epsilon_2 p^{a_2}
\bot \epsilon_3 p^{a_3}$ with $a_1 \le a_2 \le a_3$ and 
and $\epsilon_i \in \ZZ_p^\times \ (i=1,2,3)$.
Suppose that $T$ is non-maximal. Then, by Lemma  \ref{lem.maximal-element3}, we have $a_1 \ge 1$, or $a_3 \ge 2$, or $a_1=0,a_2=a_3=1$ and $\chi_p(-\epsilon_2\epsilon_3)=1$.
Suppose that $\beta_p(U_0 \bot pU_0,T) \not=0$. Then, there is a primitive matrix $X=(x_{ij})_{1 \le i \le 4, 1 \le j \le 3} \in M_{4,3}(\ZZ_p)$ such that
\begin{align*}
(U_0 \bot pU_0)[X] \equiv T \pmod{p^e \calh_3(\ZZ_p)}
\tag{a}
\end{align*}
for an integer $e \ge 2$.
In the case $a_1 \ge 1$, (a) implies that  ${\bf x}_j=(x_{i,j})_{1 \le i \le 2}$ is primitive for some $j=1,2,3$ and that
\[U_0 [{\bf x}_j]\equiv 0 \pmod{p}.\]
This is impossible because $\chi_p(\det U_0)=-1$.
In the case $a_3 \ge 2$, (a) implies that 
${\bf y}=(x_{i,3})_{1 \le i \le 4}$ is primitive, and that
\[(U_0 \bot pU_0) [{\bf y}]\equiv 0 \pmod{p^2}.\]
This is also impossible by the same reason as above.
In the case, $a_1=0,a_2=a_3=1$ and $\chi_p(-\epsilon_2\epsilon_3)=1$, (a) implies that ${\bf z}=(x_{ij})_{3 \le i \le 4, 2 \le j \le 3}$ belongs to $GL_2(\ZZ_p)$ and
\[pU_0[{\bf z}] \equiv \epsilon_2 p \bot \epsilon_3 p \pmod{p^2 \calh_2(\ZZ_p)}.\]
This is also impossible because $\chi_p(-\det (pU_0)) \not=\chi_p(-\det (\epsilon_2 p \bot \epsilon_3 p))$.
This proves the assertion (1).
Next suppose that $T$ is maximal and $\eta_p(T)=1$. Then, again by Lemma  \ref{lem.maximal-element3}, $a_1=a_2=a_3=0$, or $a_1=a_2=0,a_3=1$ and $\chi_p(-\epsilon_1\epsilon_2)=1$.
In the first case, since the $\ZZ_p/p\ZZ_p$-rank of $U_0 \bot pU_0$ is smaller than that of $T$, clearly we have $\beta_p(U_0 \bot pU_0,T)=0$. In the second case, since $\chi_p(-\det U_0) \not=\chi_p(-\det (\epsilon_1 \bot \epsilon_2))$, we also have 
$\beta_p(U_0 \bot pU_0,T)=0$. Finally suppose that
$T$ is maximal and $\eta_p(T)=-1$. Let $X$ be a primitive matrix satisfying the condition (a).
First let 
$T=\epsilon_1 \bot \epsilon_2 \bot \epsilon_3 p$ with $\chi_p(-\epsilon_1\epsilon_2)=-1$. 
Write $X=\begin{pmatrix} X_{11} & X_{12} \\ X_{21} & X_{22} \end{pmatrix}$ with $X_{11} , M_{21} \in M_{2}(\ZZ_p)$ and $X_{12},X_{22} \in M_{21}(\ZZ_p)$. Then we have
\begin{align*}
U_0[X_{11}] +pU_0[X_{21}] \equiv \epsilon_1 \bot \epsilon_2 \pmod{p^e}, \tag{b11}
\end{align*} 
\begin{align*}
{}^tX_{11}U_0X_{12}+{}^tX_{21}pU_0X_{22} \equiv 0 \pmod{p^e}, \tag{b12}
\end{align*}
\begin{align*}
U_0[X_{12}]+pU_0[X_{22}] \equiv \epsilon_3 p \pmod{p^e}.\tag{b22}
\end{align*}
Since $U_0$ and $\epsilon_1 \bot \epsilon_2$ are invertible in $M_2(\ZZ_p)/p^eM_2(\ZZ_p)$, so is $X_{11}$ by (b11). 
Hence, by (b12), we have
\[X_{12}\equiv -p({}^tX_{11}U_0)^* \ {}^t\!X_{21}U_0X_{22} \pmod{p^e},\]
and by (b22), we have
\begin{align*}
p(pU_0[({}^tX_{11}U_0)^*\ {}^t\!X_{21}U_0]+U_0)[X_{22}] \equiv \epsilon_3 p \pmod{p^e}, \tag{b$22'$}
\end{align*}
where $({}^tX_{11}U_0)^*$ is the inverse of ${}^tX_{11}U_0$ in $M_2(\ZZ_p)/p^eM_2(\ZZ_p)$.
 For $X_{21} \in M_2(\ZZ_p)/p^eM_2(\ZZ_p)$ put 
\[\calc_e(X_{21})=\calb_e(U_0,-pU_0[X_{21}]+(\epsilon_1 +\epsilon_2)),\]
and for $X_{21} \in M_2(\ZZ_p)/p^eM_2(\ZZ_p)$ and $X_{11} \in \calc_e(X_{21})$ put
\[\calc_e(X_{21},X_{11})=\calb_e(p(pU_0[({}^tX_{11}U_0)^*\ {}^t\!X_{21}U_0]+U_0),\epsilon_3 p).\]
Then, by (b11) and (b$22'$), we have
\begin{align*}
&\#\calb_e(U_0 \bot pU_0,T) \\
&=\sum_{X_{21} \in M_2(\ZZ_p)/p^eM_2(\ZZ_p)}  \sum_{X_{11} \in \calc_e(X_{21})} \# \calc_e(X_{21},X_{11}).
\end{align*}
For  non-negative integers $e' \le e$, let 
$$\pi_{e'}^e:M_{rs}(\ZZ_p)/p^eM_{rs}(\ZZ_p) \longrightarrow  
M_{rs}(\ZZ_p)/p^{e'}M_{rs}(\ZZ_p) $$
 be the natural projection.  For $Y \in M_{rs}(\ZZ_p)/p^eM_{rs}(\ZZ_p)$, we use the same symbol $Y$ to denote the element $\pi_{e'}^e(Y) \in M_{rs}(\ZZ_p)/p^{e'}M_{rs}(\ZZ_p)$. Then $\pi_{e'}^e$ induces a mapping from $\calc_e(X_{21})$ to $\calc_{e'}(X_{21})$, which will also be denoted by $\pi_{e'}^e$. 
We have 
\[\calc_{1}(X_{21})=\calb_1(U_0,\epsilon_1 \bot \epsilon_2)\]
for any integer $e \ge 1$ and $X_{21} \in M_2(\ZZ_p)/p^eM_2(\ZZ_p)$.
Hence, by the $p$-adic Newton approximation method, we easily see that 
\[\#\calc_e(X_{21})=p^{e-1}\#\calc_1(X_{21})=p^{e-1}\#\calb_1(U_0,\epsilon_1 \bot \epsilon_2),\]
and in particular
\[\#\calb_e(U_0,\epsilon_1 \bot \epsilon_2)=p^{e-1}\#\calb_1(U_0,\epsilon_1 \bot \epsilon_2).\]
This implies that 
\[p^{-1}\#\calb_1(U_0,\epsilon_1 \bot \epsilon_2)=\beta_p(U_0,\epsilon_1 \bot \epsilon_2),\]
and 
\[\#\calc_e(X_{21})=p^e\beta_p(U_0,\epsilon_1 \bot \epsilon_2).\]
Moreover, we have 
\[\calc_2(X_{21},X_{11})=\calb_2(pU_0,\epsilon_3 p)\]
for any integer $e \ge 2$, $X_{21} \in M_2(\ZZ_p)/p^eM_2(\ZZ_p)$ and $X_{11} \in \calc_e(X_{21})$. Therefore, in a way similar to above, we have
\[\#\calc_e(X_{21},X_{11})=p^e\beta_p(pU_0,\epsilon_3 p).\]
Hence,  we have 
\begin{align*}
&\#\calb_e(U_0 \bot pU_0,T)\\
& =p^{2e}\beta_p(U_0,\epsilon_1 \bot \epsilon_2)\beta_p(pU_0,\epsilon_3p)\sum_{X_{21} \in M_2(\ZZ_p)/p^eM_2(\ZZ_p)} 1\\
&=p^{6e} \beta_p(U_0,\epsilon_1 \bot \epsilon_2)\beta_p(pU_0,\epsilon_3 p).
\end{align*} 
Therefore,
 we have
\[\beta_p(U_0 \bot pU_0,T)=\beta_p(U_0,\epsilon_1 \bot \epsilon_2)\beta_p(pU_0,\epsilon_3 p).\]
We note that $\beta_p(pU_0,\epsilon_3 p)=p\beta_p(U_0,\epsilon_3 )
=p(1+p^{-1})$. Thus the assertion follows from  \cite[Theorem 5.6.3]{Ki2}. Next let $T=\epsilon_1 \bot \epsilon_2p \bot \epsilon_3 p$ with $\chi_p(-\epsilon_2\epsilon_3)=-1$. Then, in the same way as above, we have
\[\beta_p(U_0 \bot pU_0,T)=\beta_p(U_0,\epsilon_1)  \beta_p(pU_0,\epsilon_2 p \bot \epsilon_3 p).\]
Thus the assertion follows from \cite[Theorem 5.6.3]{Ki2}. 
\end{proof}
A non-degenerate ${m \times m}$ matrix $D=(d_{ij})$ with entries in $\ZZ_q$ is said to be {\it reduced} if $D$ satisfies the following two conditions:\begin{enumerate}
\item[(R-1)] For 
$i=j$, $d_{ii}=q^{e_{i}}$ with a non-negative integer $e_i$; \vspace*{1mm}
 
\item[(R-2)] For 
$i\ne j$, $d_{ij}$ is a non-negative integer satisfying 
$ d_{ij} \le q^{e_j}-1$ if $i <j$ and $d_{ij}=0$ if $i >j$. 
\end{enumerate}
It is well known that  we can take the set of all reduced matrices as a  complete set of representatives of 
$GL_m(\ZZ_q) \backslash M_m(\ZZ_q)^{\rm nd}.$


\begin{lemma}\label{lem.uniquenes-of-maximal-element3} 
Let $q$ be an odd prime number. 
Let $T \in \calh_3(\ZZ_q)^{\rm nd}$ and suppose that $\eta_q(T)=-1$. 
Then there is a unique element $g \in GL_3(\ZZ_q) \backslash M_3(\ZZ_q)^{\rm nd}$ 
such that $T[g^{-1}]$ is maximal.
\end{lemma}
\begin{proof}
Since $\eta_q(T)=-1$, without loss of generality we assume 
that
$T=\epsilon_1 q^{a_1} \bot \epsilon_2 q^{a_2} \bot \epsilon_3 q^{a_3}$
with $a_1$ even, and $a_2$ or $a_3$ is odd.
 Let $g \in M_3(\ZZ_q)^{\rm nd}$ such that $T[g^{-1}]$ is maximal.
We may assume that
\[g=\begin{pmatrix} q^{e_1} & d_{12} & d_{13} \\
0 & q^{e_2} & d_{23} \\ 0 & 0 & q^{e_3}
\end{pmatrix}\]
satisfies the conditions (R-1) and (R-2).
Then we have 
\begin{align*}
&T[g^{-1}]=\\
&{\footnotesize \begin{pmatrix}
\epsilon_1q^{a_1-2e_1} & q^{-e_1}\epsilon_1q^{a_1} d_{12}^*& q^{-e_1}\epsilon_1q^{a_1}d_{13}^*\\
q^{-e_1}\epsilon_1q^{a_1} d_{12}^* & \epsilon_1q^{a_1}(d_{12}^*)^2+ \epsilon_2q^{a_2-2e_2}& d_{12}^*\epsilon_1q^{a_1}d_{13}^*+q^{-e_2}\epsilon_2q^{a_2}d_{23}^*\\
d_{13}^*\epsilon_1q^{a_1-e_1} & d_{13}^*\epsilon_1q^{a_1}d_{12}^*+q^{-e_2}\epsilon_2q^{a_2}d_{23}^* &
\epsilon_1q^{a_1}(d_{13}^*)^2+\epsilon_2 q^{a_2}(d_{23}^*)^2+\epsilon_3 q^{a_3-2e_3}
\end{pmatrix}},
\end{align*}
where 
\vspace{1mm}
\\
$d_{12}^*=-q^{-e_1-e_2}d_{12},d_{13}^*=q^{-e_1-e_2-e_3}(d_{12}d_{23}-q^{e_2}d_{13})$ and
$d_{23}^*=-q^{-e_2-e_3}d_{23}.$
\vspace{1mm}
\\
First assume $a_2$ is even and $a_3$ is odd.
Since $T[g^{-1}]$ is maximal, we have $e_1=a_1/2$, and hence 
$q^{-e_1}\epsilon_1q^{a_1} d_{12}^*=-\epsilon_1q^{-e_2}d_{12}$, Then, by (R-2), we have $d_{12}=0$. Hence, we have $q^{-e_1}\epsilon_1q^{a_1}d_{13}^*=-\epsilon_1q^{-e_3}d_{13}$, and again by (R-2), we have $d_{13}=0$.
We also have $\epsilon_1q^{a_1}(d_{12}^*)^2+ \epsilon_2q^{a_2-2e_2}=\epsilon_2q^{a_2-2e_2}$, and by the maximality condition, we have $e_2=a_2/2$ and again by (R-2), we have $d_{23}=0$. This proves the uniqueness of $g$.
Next assume that $a_2$ and $a_3$ are odd. Since we have $\eta_q(T)=-1$, we have $\chi_q(-\epsilon_2\epsilon_3)=-1$. 
In the same way as above, we can prove the uniqueness of $e_1=a_1/2$ and $d_{12}=d_{13}=0$. Then, $\epsilon_1q^{a_1}(d_{12}^*)^2+ \epsilon_2q^{a_2-2e_2}=\epsilon_2q^{a_2-2e_2}$, and by the maximality condition, we have $e_2=(a_2-1)/2$.
Then, we have $d_{12}^*\epsilon_1q^{a_1}d_{13}^*+q^{-e_2}\epsilon_2q^{a_2}d_{23}^*=-\epsilon_2q^{-e_3+1}d_{23}$, and 
therefore, by (R-2), we have  $-q^{-e_3+1}d_{23} \in \ZZ_p^\times$ if $d_{23} \not=0$. 
Hence, we have 
\[\epsilon_1q^{a_1}(d_{13}^*)^2+\epsilon_2 q^{a_2}(d_{23}^*)^2+\epsilon_3 q^{a_3-2e_3}=q^{-1}\epsilon_2(q^{-e_3+1}d_{23})^2+\epsilon_3q^{a_3-2e_3},\]
and it does not belong to $\ZZ_q$ because $\chi_q(-\epsilon_2\epsilon_3)=-1$. This implies that $d_{23}=0$ and the assertion holds.
\end{proof}
\noindent
{\bf Proof of Proposition \ref{prop.local-density-at-p} \,(1).}
The assertion follows from Lemmas 
\ref{lem.local-density-via-primitive-local-density}, \ref{lem.explicit-primitive-locla-density3}, and
\ref{lem.uniquenes-of-maximal-element3}.

Next we prove Proposition \ref{prop.local-density-at-p} (2).
The following lemma is easy to prove.

\begin{lemma} \label{lem.maximal-element4}
Let $q$ be an odd prime number.
Suppose that $T \in \calh_4(\ZZ_q)^{\rm nd}$ is  maximal. Then
$T$ is one of the following matrices.
\begin{itemize}
\item[{\rm (1)}] $T \sim_{GL_4(\ZZ_q)} \epsilon_1 \bot \epsilon_2 \bot \epsilon_3 \bot \epsilon_4$ with $\epsilon_i \in \ZZ_q^\times \ (i=1,2,3,4)$.
\item[{\rm (2)}] $T \sim_{GL_4(\ZZ_q)} \epsilon_1 \bot \epsilon_2 \bot \epsilon_3 \bot  q\epsilon_4$ with $\epsilon_i \in \ZZ_q^\times \ (i=1,2,3,4)$.
\item[{\rm (3)}] $T \sim_{GL_4(\ZZ_q)} \epsilon_1 \bot \epsilon_2 \bot q\epsilon_3 \bot q\epsilon_4$ with $\epsilon_i \in \ZZ_q^\times \ (i=1,2,3,4)$ such that \\ $\chi_q(-\epsilon_3\epsilon_4)=-1$.
\end{itemize}
Moreover, in case {\rm (3)}, we have
$\eta_q(T)=-1 $.
\end{lemma}
\begin{lemma}\label{lem.explicit-primitive-locla-density4}
Let $T \in \calh_4(\ZZ_p)^{\rm nd}$ be maximal.
If $\alpha_p(U_0 \bot pU_0,T) \not=0$, then 
$$T \sim_{GL_4(\ZZ_p)} \epsilon_1 \bot \epsilon_2 \bot p\epsilon_3 \bot p\epsilon_4$$
such that $\epsilon_1 \epsilon_2\epsilon_3 \epsilon_4 \in (\ZZ_p^\times)^2$ and $\chi_p(-\epsilon_1\epsilon_2)=\chi_p(-\epsilon_3\epsilon_4)=-1$. 
Conversely, if $T$ satisfies these conditions, then 
\begin{align*}
\alpha_p(U_0 \bot pU_0, T)=2(1+p^{-1})^2 p.
 \end{align*}
\end{lemma}
\begin{proof}
By the assumption, $\det (2T)$ is divided by $p^2$, and $T$ must be the case (3) of Lemma \ref{lem.maximal-element4}, $T$. 
Moreover, in this case we have $p^{-2}\det (2T)=\prod_{i=1}^4 \epsilon_i=\epsilon^2$ with $\epsilon \in \ZZ_p^\times$. Hence, we have $\chi_p(-\epsilon_1\epsilon_2)=\chi_p(-\epsilon_3\epsilon_4)=-1$. 
Hence, $T \sim_{GL_4(\ZZ_p)} U_0 \bot pU_0$, and the second assertion follows from \cite[Theorem 6.8.1]{Ki2}.
\end{proof}
\begin{lemma}\label{lem.uniquenes-of-maximal-element4} Let $q$ be an odd prime number. 
Let $T \in \calh_4(\ZZ_q)^{\rm nd}$ and suppose that $\eta_q(T)=-1$. Then there is a unique element $g \in GL_4(\ZZ_q) \backslash M_4(\ZZ_q)^{\rm nd}$ such that $T[g^{-1}]$ is maximal.
\end{lemma}
\begin{proof}
The assertion can be proved in the same manner as Lemma  \ref{lem.uniquenes-of-maximal-element3}.
\end{proof}
\noindent
{\bf Proof of Proposition \ref{prop.local-density-at-p}\,(2).}
The assertion follows from 
Lemmas \ref{lem.local-density-via-primitive-local-density}, \ref{lem.explicit-primitive-locla-density4}, 
and \ref{lem.uniquenes-of-maximal-element4}.
\vspace{3mm}
\\
{\bf Proof of Theorem \ref{th.main-result}.}

(1) If $\eta_p(T)=1$, by Theorem \ref{th.limit-of-FC-of-Eisenstein} (1) and Propositions \ref{prop.mass-formula} and 
\ref{prop.local-density-at-p}, we have
\[a(\widetilde E_2^{(3)},T)=a(\mathrm{genus} \ \Theta^{(3)}(S^{(p)}),T)=0.\]
Suppose that $\eta_p(T)=-1$. Then, we have
\[a(\mathrm{genus} \ \Theta^{(3)}(S^{(p)}),T)=16 (1+p)^2 p^{-4} \pi^4 \prod_{q \not=p} (1-q^{-2})^2 \prod_{q \not=p} F_q(T,q^{-2}).\]
We have 
\[\prod_{q \not=p} (1-q^{-2})^2=\zeta(2)^{-2}(1-p^{-2})^{-2}=\frac{36}{\pi^4(1-p^{-2})^2}.\]
Hence, we have
\[a(\mathrm{genus} \ \Theta^{(3)}(S^{(p)}),T)=\frac{576}{(p-1)^2} \prod_{q \not=p} F_q(T,q^{-2}).\]
This coincides with $a(\widetilde E_2^{(3)},T)$ in view of Theorem \ref{th.limit-of-FC-of-Eisenstein} (1).

(2) If $p^{-2}\det (2T)$ is not a square integer or $\eta_p(T)=1$, by Theorem \ref {th.limit-of-FC-of-Eisenstein} (2), and Propositions \ref{prop.mass-formula} and 
\ref{prop.local-density-at-p}, we have
\[a(\widetilde E_2^{(4)},T)=a(\mathrm{genus} \ \Theta^{(4)}(S^{(p)}),T)=0.\]
Suppose that $p^{-2}\det (2T)$ is a square integer and $\eta_p(T)=-1$. Then, we have
\[a(\mathrm{genus} \ \Theta^{(4)}(S^{(p)}),T)=32 (1+p)^2 p^{-4} \pi^4 \prod_{q \not=p} (1-q^{-2})^2 \prod_{q \not=p} F_q(T,q^{-3}).\]
We have 
\[\prod_{q \not=p} (1-q^{-2})^2=\zeta(2)^{-2}(1-p^{-2})^{-2}=\frac{36}{\pi^4(1-p^{-2})^2}.\]
Hence, we have
\[a(\mathrm{genus} \ \Theta^{(4)}(S^{(p)}),T)=\frac{1152}{(p-1)^2} \prod_{q \not=p} F_q(T,q^{-3}).\]
This coincides with $a(\widetilde E_2^{(4)},T)$ in view of Theorem \ref{th.limit-of-FC-of-Eisenstein} (2).

These complete the proof of Theorem \ref{th.main-result}.

\subsection{General case}
To prove the main theorem, we must consider the case $n\geq 5$. For the Siegel operator $\Phi$,
we have
$$
\begin{cases}
\Phi (\widetilde{E}_2^{(n)})=\widetilde{E}_2^{(n-1)},\\
\Phi (\text{genus}\,\Theta^{(n)}(S^{(p)}))=\text{genus}\,\Theta^{(n-1)}(S^{(p)}).
\end{cases}
$$
From this, it is sufficient to prove the following proposition.
\begin{proposition}
\label{n5}
\quad Assume that $n\geq 5$. For any $T\in\Lambda_n^+$, we have
$$
(*)\qquad\qquad \qquad a(\widetilde{E}_2^{(n)},T)=a({\rm genus}\,\Theta^{(n)}(S^{(p)}),T)=0.
$$
\end{proposition}
\begin{proof}
First we consider ${\rm genus}\,\Theta^{(n)}(S^{(p)})$.
Since $S^{(p)}\in\Lambda_4^+$ and $T\in\Lambda_n^+$\;$(n\geq 5)$,
$a(\theta^{(n)}(S_i;Z),T)=0$ holds for each theta series $\theta^{(n)}(S_i;Z)$. Hence, we obtain
$$
a({\rm genus}\,\Theta^{(n)}(S^{(p)}),T)=0.
$$
Next we investigate $\widetilde{E}_2^{(n)}$ and prove
$$
\label{vanishE2}
(**)\qquad\qquad\qquad \lim_{m\to\infty}a(E_{k_m}^{(n)},T)=a(\widetilde{E}_2^{(n)},T)=0
$$
for any $T\in\Lambda_n^+$. 
We recall the formula for $a(E_k^{(n)},T)$
given in Proposition \ref{prop.FC-Siegel}.

We extract the following factor in the formula for  $a(E_k^{(n)},T)$:
$$
A_{k,n}(T):=\frac{k}{B_k}\cdot\prod_{i=1}^{[n/2]}\frac{k-i}{B_{2k-2i}}\cdot
\begin{cases}
\frac{B_{k-n/2,\chi_T}}{k-n/2} & (\text{if  $n$ is even}),\\
1 &  (\text{if  $n$ is odd}).
\end{cases}
$$
To prove $(**)$,  it suffices to show that
$$
(\ddagger)\qquad\qquad\qquad \lim_{m\to\infty}A_{k_m,n}(T)=0\quad (\text{$p$-adically}).
$$

First we assume that $p>3$.
By Kummer's congruence for Bernoulli numbers, the factors
$$
\frac{k_m}{B_{k_m}}\;\;\text{and}\;\;\frac{k_m-i}{B_{2k_m-2i}}\;\;(1\leq i \leq [n/2]
\;\;\text{with}\;\;
i \not\equiv 2 \pmod{(p-1)})
$$
in $A_{k_m,n}(T)$ have $p$-adic limits when $m\longrightarrow \infty$. 
\vspace{2mm}
\\
We forcus our attention on the factors
$$
\frac{k_m-i}{B_{2k_m-2i}}\;\;\text{for}\;\; i\;\;\text{with}\quad i \equiv 2 \pmod{(p-1)}.
$$
In these cases, by the von Staudt-Clausen theorem, we obtain
$$
\text{ord}_p\left(\frac{k_m-i}{B_{2k_m-2i}}\right)\geq 1.
$$
In particular, in the case of  $i=2\leq [n/2]$, the following identity holds:
$$
\text{ord}_p\left(\frac{k_m-2}{B_{2k_m-4}}\right)=m.
$$
( Such $i$ also appears when $n=4$. See Remark \ref{cacellation}.)  Consequently, there is a constant $C$ such that
$$
\text{ord}_p\left(\frac{k_m}{B_{k_m}}\cdot\prod_{i=1}^{[n/2]}\frac{k_m-i}{B_{2k_m-2i}}\right)\geq 
C+m
$$
for sufficiently large $m$. Regarding the factor $B_{k_m-n/2,\chi_T}/(k_m-n/2)$, we use
Carlitz's result (\cite{Carlitz2}) for the generalized Bernoulli numbers in the case of quadratic
Dirichlet characters. We have
\begin{align*}
\text{ord}_p\left( \frac{B_{k_m-n/2,\chi_T}}{k_m-n/2}\right)
                                         & =\text{ord}_p\left(B_{k_m-n/2,\chi_T}\right)
                                                                              -\text{ord}_p\left(k_m-n/2\right)\\
                                         & \geq -1-\text{ord}_p\left(2-n/2\right),
\end{align*}
for sufficiently large $m$. By assumption $n\geq 5$, the values 
$$
\text{ord}_p\left( \frac{B_{k_m-n/2,\chi_T}}{k_m-n/2}\right)
$$
have a lower bound for sufficiently large $m$.
\vspace{1mm}
\\
Combining these results, we can prove $(\ddagger)$ when $p>3$.

Next we consider the case $p=3$. By von Staudt-Clausen theorem, we obtain
$$
\text{ord}_3\left(\frac{k_m}{B_{k_m}}\right)=1,\quad
\text{ord}_3\left(\frac{k_m-i}{B_{2k_m-2i}}\right)=\text{ord}_3(k_m-i)+1\geq 1
\quad (1\leq i\leq [n/2]).
$$
In particular, when $i=2\geq [n/2]$, the following identity holds: 
$$
\text{ord}_3\left(\frac{k_m-2}{B_{2k_m-4}}\right)=m.
$$
Since we know that $\text{ord}_3(B_{k_m-n/2,\chi_T}/(k_m-n/2))$ has a lower bound as in the case
$p>3$, the statement $(\ddagger)$ is also proven in the case $p=3$.

This completes the proof of Proposition \ref{n5}.
\end{proof}
\noindent
\begin{remark}
\label{cacellation}
\quad An exceptional factor $(k_m-2)/\,B_{2k_m-4}$  in the product
\\
 $\prod (k_m-i)/\,B_{2k_m-2i}$
also appears when $n=4$. However, in this case, cancellation occurs between
$$
\frac{k_m-2}{B_{2k_m-4}}\quad \text{and}\quad \frac{B_{k_m-2,\chi_T}}{k_m-2}.
$$
\vspace{2mm}
\end{remark}
By combining Corollary \ref{cor.main-result2} and Proposition \ref{n5}, we have
proved our main result, Theorem \ref{statementmain}.
\section{Applications}
\subsection{Modular forms on $\boldsymbol{\Gamma_0^{(n)}(p)}$}
In \cite{Serre}, Serre proved the following result.
\begin{theorem}
{\rm (Serre \cite{Serre})}\quad Let $p$ be an odd prime number and $\mathbb{Z}_{(p)}$ the local
ring consisting of $p$-integral rational numbers. For any $f\in M_2(\Gamma^{(1)}_0(p))_{\mathbb{Z}_{(p)}}$,
there is a modular form $g\in M_{p+1}(\Gamma^{(1)})_{\mathbb{Z}_{(p)}}$ satisfying
$$
f \equiv g \pmod{p}.
$$
\end{theorem}
An attempt to generalize this result to the case of Siegel modular forms can be found in \cite{BN2}.
\\
Here we consider the first $p$-adic approximation of $\widetilde{E}_2^{(n)}$, that is,
$$
E_{k_1}^{(n)}=E_{p+1}^{(n)}.
$$
\begin{theorem}
Let $p$ be a prime number such that $p>n$. The modular form
$\widetilde{E}_2^{(n)}\in M_2(\Gamma_0^{(n)}(p))_{\mathbb{Z}_{(p)}}$ is
congruent to $E_{p+1}^{(n)}\in M_{p+1}(\Gamma^{(n)})_{\mathbb{Z}_{(p)}}$ mod $p$:
$$
\widetilde{E}_2^{(n)} \equiv E_{p+1}^{(n)} \pmod{p}.
$$
\end{theorem}
The above result provides an example of Serre's type congruence in the case of Siegel
modular forms. ( For $n=2$, this theorem has already been proved in \cite[\sc Propositon 4]{KN}.
The $p$-integrality of $\widetilde{E}_2^{(n)}$ comes from the explicit
formula for the Fourier coefficients.)
\subsection{Theta operators}
\label{ThetaOp}
For a Siegel modular form $\displaystyle F=\sum a(F,T)q^T$, we define
$$
\varTheta (F):=\sum a(F,T)\cdot\text{det}(T)\,q^T\in\mathbb{C}[q_{ij}^{-1},q_{ij}][\![q_1,\ldots,q_n]\!].
$$
The operator $\varTheta$ is called the {\it theta operator}. This operator was first studied by
Ramanujan in the case of elliptic modular forms, and the generalization to the case of
Siegel modular forms can be found in \cite{BN1}.

If a Siegel modular form $F$ satisfies
$$
\varTheta (F) \equiv 0 \pmod{N},
$$
we call it an element of the space of the {\it mod $N$ kernel of the theta operator}. For example,
Igusa's cusp form $\chi_{35}\in M_{35}(\Gamma^{(2)})_{\mathbb{Z}}$ satisfies the congruence
relation
$$
\varTheta (\chi_{35}) \equiv 0 \pmod{23}\qquad\qquad (\text{cf. \cite{K-K-N}}),
$$
namely, $\chi_{35}$ is an element of the space of mod $23$ kernel of the theta operator.
\begin{theorem}
\label{modpaquare}
Assume that $p\geq 3$. Then we have
$$
\varTheta (E_{p+1}^{(3)}) \equiv 0 \pmod{p},\qquad
\varTheta (E_{p^2-p+2}^{(4)}) \equiv 0 \pmod{p^2}.
$$
The second congruence shows that the Siegel Eisenstein series $E_{p^2-p+2}^{(4)}$ is 
an element of the mod $p^2$ kernel of the theta operator.
\end{theorem}
\begin{proof}
To prove the first congruence relation, we consider the first approximation of $\widetilde{E}_2^{(3)}$:
$$
\widetilde{E}_2^{(3)} \equiv E_{2+p-1}^{(4)}=E_{p+1}^{(3)} \pmod{p}.
$$
If $T\in\Lambda^+_3$ satisfies $a({\rm genus}\,\Theta^{(3)}(S^{(p)}),T)=a(\widetilde{E}_2^{(3)},T)\ne 0$,
then, by Lemma \ref{lem.explicit-primitive-locla-density3}, we have $\text{det}(2T) \equiv 0 \pmod{p}$.
This fact implies that
$$
\varTheta (E_{p+1}^{(3)}) \equiv \varTheta(\widetilde{E}_2^{(3)})
= \varTheta ({\rm genus}\,\Theta^{(3)}(S^{(p)}) )\equiv 0 \pmod{p}.
$$
We consider the second congruence relation.
Considering the second $p$-adic approximation of $\widetilde{E}_2^{(4)}$, we obtain
$$
\widetilde{E}_2^{(4)} \equiv E_{2+(p-1)p}^{(4)}=E_{p^2-p+2}^{(4)} \pmod{p^2}.
$$
Therefore, it is sufficient to prove that
$$
\varTheta (\widetilde{E}_2^{(4)}) \equiv 0 \pmod{p^2}.
$$
Assume that $a(\widetilde{E}_2^{(4)},T)\ne 0$ for $T\in\Lambda_4^+$. Then $\text{det}(2T)$
is square by Theorem \ref{th.limit-of-FC-of-Eisenstein}, (2). Under this condition, 
if we further
assume that $\text{det}(2T)\not\equiv 0 \pmod{p}$, then we have 
$a(\widetilde{E}_2^{(4)},T)=0$ by Theorem \ref{th.limit-of-FC-of-Eisenstein}, (2.2).
This is a contradiction. Therefore, we have $\text{det}(2T) \equiv 0 \pmod{p}$, equivalently,
 $\text{det}(2T) \equiv 0 \pmod{p^2}$. This means that, 
 if $a(\widetilde{E}_2^{(4)},T)\ne 0$ for $T\in\Lambda_4^+$, then $T\in\Lambda_4^+$ satisfies
 $\text{det}(2T) \equiv 0 \pmod{p^2}$. This implies
 $$
 \varTheta (\widetilde{E}_2^{(4)}) \equiv 0 \pmod{p^2}
 $$
and  completes the proof of Theorem \ref{modpaquare}.
\end{proof}
\noindent
\begin{remark}
The first congruence relation in Theorem \ref{modpaquare} can also be shown as a special
case of the main theorem in \cite[Theorem 2.4]{N-T}.
\end{remark}
\subsection{Numerical examples}
In this section, we provide examples of Fourier coefficients of $\widetilde{E}_2^{(n)}$
and \\
genus\,$\Theta^{(n)}(S^{(p)})$, which certify the validity of our identity in our main result.
\vspace{2mm}
\\
\fbox{\textbf{Case $\boldsymbol{n=3:}$}}
\vspace{2mm}
\\
We take
$$
p=11,\quad T={\scriptsize \begin{pmatrix} 1 & 0 & \tfrac{1}{2} \\
                                             0 & 1 &  0              \\
                                             \tfrac{1}{2} & 0 & 3 \end{pmatrix}}
                                             \in \Lambda_3^+\quad\text{with}\quad\text{det}(T)=11/4,
$$
and calculate
$a(\widetilde{E}_2^{(3)},T)$ and $a(\text{genus}\,\Theta^{(3)}(S^{(p)}),T)$.
\vspace{2mm}
\\
\textbf{Calculation of} $\boldsymbol{a(\widetilde{E}_2^{(3)},T)}$:
\vspace{2mm}
\\
By Theorem \ref{th.limit-of-FC-of-Eisenstein}, (1),
$$
a(\widetilde{E}_2^{(3)},T)=\frac{576}{(1-11)^2}\cdot\lim_{m\to\infty}(1-11^{10\cdot 11^{m-1}})=
\frac{144}{25}.
$$
\textbf{Calculation of} $\boldsymbol{a({\rm genus}\,\Theta^{(3)}(S^{(11)}),T)}$:
\vspace{2mm}
\\
We can take three representatives of $GL_4(\mathbb{Z})$-equivalence classes in genus$(S^{(11)})$:
$$
{\scriptsize
S_1^{(11)}:=\begin{pmatrix}         1     & 0 & \tfrac{1}{2} & 0 \\ 0 &        1       & 0 & \tfrac{1}{2} \\
                                  \tfrac{1}{2} & 0 &       3       & 0 \\ 0  & \tfrac{1}{2} & 0 &       3
              \end{pmatrix},\;
S_2^{(11)}:=\begin{pmatrix}     1     & \tfrac{1}{2} & \tfrac{1}{2} & \tfrac{1}{2} \\ \tfrac{1}{2} & 1 & 0 & \tfrac{1}{2} \\
                                  \tfrac{1}{2} & 0 &  4  & 2 \\ \tfrac{1}{2} & \tfrac{1}{2} & 2 &   4
              \end{pmatrix},\;
S_3^{(11)}:=\begin{pmatrix}     2   & 1 & \tfrac{1}{2} & \tfrac{1}{2} \\ 1 & 2 & 0 & \tfrac{1}{2} \\
                                  \tfrac{1}{2} & 0 &  2  & 1 \\ \tfrac{1}{2} & \tfrac{1}{2} & 1 &   2
              \end{pmatrix}.  
 }                                     
$$
Moreover we have
$$
a(S_1^{(11)},S_1^{(11)})=32,\quad a(S_2^{(11)},S_2^{(11)})=72,
\quad a(S_3^{(11)},S_3^{(11)})=24.\qquad (\text{cf.\;\cite{Nipp}}).
$$
Therefore,
\begin{align*}
& a(\text{genus}\,\Theta^{(3)}(S^{(11)}),T)=\\
&[a(\theta^{(3)}(S_1^{(11)};Z),T)/32+a(\theta^{(3)}(S_2^{(11)};Z),T)/72+a(\theta^{(3)}(S_3^{(11)};Z),T)/24]\\
& \cdot [(1/32)+(1/72)+(1/24)]^{-1}.
\end{align*}
Direct calculations show that
$$
a(\theta^{(3)}(S_1^{(11)};Z),T)=16,\qquad
a(\theta^{(3)}(S_2^{(11)};Z),T)=a(\theta^{(3)}(S_3^{(11)};Z),T)=0
$$
for the above $T\in\Lambda_3^+$.
Hence, we have
$$
a(\text{genus}\,\Theta^{(3)}(S^{(11)}),T)=16 \times\frac{9}{25}=\frac{144}{25}.
$$
Of course, this value is consistent with the value obtained using the equations
given in Propositions  \ref{prop.mass-formula}\, (1) and  \ref{prop.local-density-at-p}\, (1):
\begin{align*}
a(\text{genus}\,\Theta^{(3)}(S^{(11)}),T) & = \frac{8}{11^3}\pi^4
                                                          \cdot \alpha_{11}(U_0\,\bot\,11\cdot U_0,T)
                                                          \cdot \prod_{q\ne 11}\alpha_q(H_2,T) \\
                                                    & =  \frac{8}{11^3}\pi^4
                                                          \cdot 2(1+11)(1+11^{-1})
                                                          \cdot \frac{36}{\pi^4(1-11^{-2})^2}\\
                                                    &=\frac{144}{25}.          
\end{align*}

Further numerical examples for the case $n=3$ can be found in \cite{Okuma}.
\vspace{2mm}
\\
\fbox{\textbf{Case $\boldsymbol{n=4:}$}}
\vspace{2mm}
\\
The values $a(\text{genus}\,\Theta^{(4)}(S^{(p)}),T)$ at $T=S^{(p)}$ can be
calculated as
$$
a(\text{genus}\,\Theta^{(4)}(S^{(p)}),S^{(p})=\left(\sum_{i=1}^d\frac{1}{a(S_i,S_i)}\right)^{-1}=M(S^{(p)})^{-1}
$$
and the values $M(S^{(p)})$ for small $p$ can be found in \cite{Nipp}:
\begin{table}[hbtp]
\begin{center}
\begin{tabular}{c | cccccc}
$p$ &                                  $3$  &   $5$   &   $7$    &  $11$  &            $13$  &  $\cdots$ 
\\ 
\hline
$M(S^{(p)})^{-1}$ & $288$  &   $72$ &    $32$  &  $\frac{288}{25}$& $8$  &  $\cdots$
\end{tabular}
\end{center}
\end{table}
\vspace{1mm}
\\
These numerical data can be verified to be consistent with those obtained from the
the formula
$$
a(\widetilde{E}_2^{(4)},S^{(p)})=\frac{1152}{(p-1)^2},
$$
which is a result of Theorem  \ref{th.limit-of-FC-of-Eisenstein} (2.2).
\section{Acknowledgements}
$\bullet$\quad We thank Prof. Fumihiro Sato for valuable discussion.
\vspace{1mm}
\\
$\bullet$\quad
This work was supported by
JSPS KAKENHI: first author,  Grant-in-Aid (B) (No.16H03919) and Grant-in-Aid (C) (No. 21K03152);
 second author, Grant-in-Aid (C) ( No. 20K03547).



\begin{thebibliography}{99}


\bibitem{BN1}
B\"{o}cherer and S.~Nagaoka, 
On mod $p$ properties of Siegel modular forms.
Math. Ann. \textbf{338}(2007), 421-433.

\bibitem{BN2}
B\"{o}cherer and S.~Nagaoka, 
On Siegel modular forms of level $p$ and their properties mod $p$,
manuscripta math. \textbf{132}(2010), 501-515.



\bibitem{Carlitz1}
L.~Carlitz, Some congruences for the Bernoulli numbers,
Amer. J. Math. \textbf{75}(1953),163-172.

\bibitem{Carlitz2}
L.~Carlitz, Arithmetic properties of generalized Bernoulli numbers,
J. reine angew. Math. \textbf{202}(1959), 174-182.


\bibitem{Fr}
J.~Fresnel,
Nombres de Bernoulli et fonctions $L$ $p$-adiques,
S\'{e}minaire Delange-Pisot-Poitou, Th\'{e}orie des nombres, tome 7,
$\text{n}^\circ 2$ (1965-1967), exp. $\text{n}^\circ 14$, 1-15.

\bibitem{IK22}
 T.~Ikeda and H.~Katsurada, 
An explicit formula of a quadratic form over a non-archimedean local field, J. reine angew. Math. 
\textbf{783}(2022), 1-47.

\bibitem{K99}
H.~Katsurada,
An explicit formula for Siegel series,
Amer. J. Math. \textbf{121}-2  (1999), 415-452.
 
\bibitem{Kat-Nag}
H.~Katsurada and S.~Nagaoka, 
On some $p$-adic properties of Siegel-Eisenstein series, 
J. Number Theory \textbf{104}(2004), 100-117.

\bibitem{KN}
T.~Kikuta and S.~Nagaoka, 
On a correspondence between $p$-adic Siegel Eisenstein series and genus theta series,
Acta. Arith. \textbf{134}-2 (2008), 111-126.

\bibitem{Ki1}
Y.~Kitaoka, 
A note on local densities of quadratic forms, 
Nagoya Math. J. \textbf{92} (1984), 145-152.

\bibitem {Ki2} 
 Y.~Kitaoka, 
 Arithmetic of Quadratic Forms, Cambridge. Tracts Math. \textbf{106}, Cambridge Univ. Press, Cambridge, 1993.
 


\bibitem{K-K-N}
T.~Kikuta, H.~Kodama and S.~Nagaoka,
Note on Igusa's cusp form of weight $35$,
Rocky Mountain J. Math. \textbf{45}-3(2015), 963-972.

\bibitem{Nagaoka1}
S.~Nagaoka, 
A remark on Serre's example of $p$-adic Eisenstein series, 
Math. Z. \textbf{235}(2000), 227-250.


\bibitem{N-T}
S.~Nagaoka and S.~Takemori,
On the mod $p$ kernel of the theta operator and Eisenstein series,
J. Number Theory \textbf{188}(2018), 281-298.

\bibitem{Nipp}
G.L.~Nipp,
Quaternary Quadratic Forms, Computer generated Tables,
Springer, Berlin, 1991.

\bibitem{Okuma}
S.~Okuma,
On some properties of $p$-adic Siegel Eisenstein series of degree 3 (in Japanese),
Master thesis, Kindai University 2007.

\bibitem{SH}  F.~Sato and Y.~Hironaka,  
Local densities of representations of quadratic forms over $p$-adic
integers (the non-dyadic case), 
J. Number Theory \textbf{83} (2000),106-136.


\bibitem{Serre}
J.-P.~Serre,
Formes modulaires et fonctions z\^{e}ta $p$-adiques,
Lecture Notes in Math., \textbf{350}, Springer, Berlin, 1973,
191-268.

\bibitem {Sh1}
G.~Shimura, Euler products and Eisenstein series, CBMS Regional Conference Series in Mathematics, 
\textbf{93}(1997) AMS.


\end{thebibliography}
\end{document}